\documentclass[12pt]{amsart}

\usepackage{verbatim, amssymb, enumitem}
\usepackage{pgf,tikz,ifthen}
\usetikzlibrary{arrows}
\usetikzlibrary{calc}

\usepackage[breaklinks=true,colorlinks=true,linkcolor=green,citecolor=red,urlcolor=blue]{hyperref}

\renewcommand \a{\alpha}

\newcommand \id{\mathrm{id}}
\newcommand \br{\mathbb{R}}

\newcommand \Rn{\mathbb R^n}

\newcommand \rk{\operatorname{rk}}
\newcommand \Ker{\operatorname{Ker}}
\newcommand \Der{\operatorname{Der}}
\newcommand \End{\operatorname{End}}

\renewcommand \Im{\operatorname{Im}}

\newcommand \Span{\operatorname{Span}}
\newcommand \Tr{\operatorname{Tr}}

\newcommand \cJ{\mathcal{J}}

\newcommand \cC{\mathcal{C}}
\newcommand \cE{\mathcal{E}}

\newcommand \mU{\mathcal{U}}
\newcommand \cV{\mathcal{V}}
\newcommand \cG{\mathcal{G}}

\newcommand \ig{\mathfrak{i}}
\newcommand\ag{\mathfrak a}
\newcommand\kg{\mathfrak k}
\newcommand\g{\mathfrak g}

\newcommand\h{\mathfrak h}
\newcommand\z{\mathfrak z}

\newcommand \so{\mathfrak{so}}

\newcommand \n{\mathfrak{n}}

\newcommand \p{\mathfrak{p}}

\newcommand \ad{\operatorname{ad}}
\newcommand \Ad{\operatorname{Ad}}

\DeclareMathOperator{\Lie}{Lie}

\newcommand \<{\langle}
\renewcommand \>{\rangle}
\newcommand \ip{\<\cdot,\cdot\>}
\newcommand \ipr{(\cdot,\cdot)}

\newtheorem*{theorem*}{Theorem}

\newtheorem*{corollary*}{Corollary}
\newtheorem*{conj*}{Conjecture}
\newtheorem{lemma}{Lemma}
\newtheorem{proposition}{Proposition}
\newtheorem*{prop*}{Proposition}

\theoremstyle{definition}

\newtheorem*{definition*}{Definition}

\theoremstyle{remark}

\newtheorem*{notation*}{Notation}
\newtheorem*{algorithm*}{Algorithm}
\newtheorem*{example*}{Example}

\begin{document}

\title[GO-nilmanifolds associated with graphs]{Geodesic orbit and naturally reductive nilmanifolds associated with graphs}

\author{Y.~Nikolayevsky}
\address{Department of Mathematics and Statistics, La Trobe University, Melbourne, Australia 3086}
\email{Y.Nikolayevsky@latrobe.edu.au}
\thanks{The author was partially supported by ARC Discovery Grant DP130103485. }



\begin{abstract}
We study Riemannian nilmanifolds associated with graphs. We prove that such a nilmanifold is geodesic orbit if and only if it is naturally reductive if and only if its defining graph is the disjoint union of complete graphs and the left-invariant metric is generated by a certain naturally defined inner product.
\end{abstract}

\maketitle

\section{Introduction}
\label{s:intro}

A Riemannian manifold $(M,g)$ is called a \emph{geodesic orbit manifold} (or a manifold with homogeneous geodesics, or a GO-manifold) if any geodesic of $M$ is an orbit of a one-parameter subgroup of the full isometry group of $(M,g)$ \cite{KV} (one loses no generality by replacing the full isometry group by its connected identity component $I_0(M)$). Any connected geodesic orbit manifold is homogeneous. Examples of geodesic orbit manifolds include symmetric spaces, weekly symmetric spaces, normal and generalised normal homogeneous spaces and naturally reductive spaces. For an up-to-date account of the state of knowledge on geodesic orbit manifolds we refer the reader to \cite{Arv,Nik2017} and the bibliographies therein.

Let $(M=G/H, g)$, where $G=I_0(M)$ and $H$ is a compact subgroup of $G$, be a homogeneous Riemannian manifold and let $\g=\h\oplus \p$ be an $\Ad(H)$-invariant decomposition, where $\g=\Lie(G), \;\h=\Lie(H)$ and $\p$ is identified with the tangent space to $M$ at $eH$. The Riemannian metric $g$ is $G$-invariant and is determined by an $\Ad(H)$-invariant inner product $\ipr$ on $\p$. The manifold $(M,g)$ is called \emph{naturally reductive manifold} if an $\Ad(H)$-invariant complement $\p$ can be chosen in such a way that $([X,Y]_{\p},X) = 0$ for all $X,Y \in \p$, where the subscript $\p$ denotes the $\p$-component. For comparison, on the Lie algebra level, $g$ is geodesic orbit if and only if for any $X \in \p$ (with any choice of $\p$), there exists $Z \in \h$ such that $([X+Z,Y]_{\p},X) =0$ for all $Y \in \p$ \cite[Proposition~2.1]{KV}. It follows that any naturally reductive manifold is a geodesic orbit manifold; the converse is false when $\dim M \ge 5$.

Note that by replacing in the above definitions the group $I_0(M)$ by its subgroup $G$ which acts transitively on $(M,g)$ one gets \emph{geodesic orbit spaces} and \emph{naturally reductive spaces}. The property of $(M,g)$ to be one of these depends on the choice of $G$.

By the structural theory developed in \cite[Theorem~1.14]{G}, \cite[Section~3]{GN}, the study of general geodesic orbit manifolds can be reduced to the study of such in the following three cases: $M$ is a nilmanifold, $M$ is compact, or $M$ admits a transitive group of isometries which is semisimple of noncompact type. In particular, any geodesic orbit nilmanifold is necessarily two-step (or abelian) \cite[Theorem~2.2]{G}. Motivated by this fact we study in this paper an important class of two-step nilpotent Lie groups, the nilpotent Lie groups associated with graphs.

Let $\cG=(\cV,\cE)$ be a finite, simple, undirected graph, with the set of vertices $\cV=\{V_1, \dots, V_n\}, \; n \ge 1$, and the set of edges $\cE=\{E_1, \dots, E_m\}, \; m \ge 0$.

We write $E_\a=V_iV_j$ if the edge $E_\a$ joins the vertices $V_i$ and $V_j$, where $\a=1, \dots, m,\; 1 \le i \ne j \le n$. A Lie algebra $\n$ is said to be \emph{associated with the graph $\cG$}, if $\n$ has a basis $\{e_1, \dots, e_n, z_1, \dots, z_m\}$ such that all the vectors $z_\a$ are in the centre of $\n$, and for $1 \le i < j \le n$, we have
\begin{equation*}
[e_i,e_j] = \left\{
              \begin{array}{ll}
                z_\a, & \hbox{if $E_\a=V_iV_j$;} \\
                0, & \hbox{otherwise.}
              \end{array}
            \right.
\end{equation*}
Any such basis for $\n$ is called \emph{a standard basis}. There are many standard bases -- any one of them can be obtained from a particular one by an automorphism of $\n$ (for the description of the derivation algebra see Section~\ref{s:pre}). However, by the result of \cite{M2014} the algebra $\n$ determines the graph $\cG$ uniquely, up to isomorphism. A Lie algebra associated with a graph is two-step nilpotent (and is abelian if and only if $\cE=\varnothing$). We say that a simply connected Lie group $G$ is \emph{associated with a graph $\cG$}, if its Lie algebra is.

An inner product on the Lie algebra $\n$ associated with a graph $\cG$ is called \emph{standard} if $\n$ admits an orthonormal standard basis; we also call \emph{standard} the corresponding left-invariant metric on the Lie group $G$ associated with $\cG$.

We introduce a more general class of inner products which are obtained from a standard inner product by taking a certain orthogonal splitting of the derived algebra and rescaling the standard inner product on every subspace of that splitting. Let $(\n, \ipr)$ be a metric two-step nilpotent Lie algebra. Denote $\z = [\n, \n]$ and $\ag = \z^\perp$. Note that $\z$ lies in the centre of $\n$. We define an (injective) linear map $J:\z \to \so(\ag), \; Z \mapsto J_Z$, by
\begin{equation}\label{eq:defJZ}
(J_Z X, Y) = (Z, [X, Y]),
\end{equation}
for $X, Y \in \ag, \; Z \in \z$. Denote $\cJ = \Span(J_Z \, | \, Z \in \z)$.

Now let a graph $\cG$ be a disjoint union of complete graphs and let $\ip$ be the standard inner product defined by a standard basis $\{e_1, \dots, e_n, z_1, \dots, z_m\}$ for the Lie algebra $\n$ associated to $\cG$. Then $\z=\Span(z_1, \dots, z_m)$ and $\ag=\Span(e_1, \dots, e_n)$. Computing the operators $J_Z$ as in \eqref{eq:defJZ} for the inner product $\ip$ we obtain $\cJ = \oplus_{\mu=1}^p \so(\ag_\mu, \ip)$, where $\ag_\mu \subset \ag$ are mutually orthogonal subspaces. The bilinear form $B(K_1, K_2)=-\frac12 \Tr(K_1K_2)$ on $\cJ$ is proportional to the restriction to $\cJ$ of the Killing form of $\so(\ag, \ip)$ and satisfies $\<Z,W\>=B(J_Z, J_W)$ for $Z, W \in \z$. Consider an arbitrary decomposition of the (reductive) Lie algebra $\cJ$ into the $B$-orthogonal sum of ideals and let an operator $\Phi \in \End(\cJ)$ be the linear combination of the projections from $\cJ$ to those ideals, with positive coefficients. Define a new inner product $\ipr$ on $\n$ by $(X+Z, Y+W)=\<X,Y\> + B(\Phi J_Z, J_W)$ for $X,Y \in \ag, \; Z, W \in \z$ (note that if we choose $\Phi=\id$, then $\ipr=\ip$). Any inner product $\ipr$ constructed in this way is called \emph{semi-standard}; we also call \emph{semi-standard} the corresponding left-invariant metric on the Lie group $G$ associated with $\cG$. Note that the restriction of a semi-standard inner product to $\z$ is an invariant inner product in the sense of \cite[Definition~2.6]{G}. We give an example of a semi-standard inner product in Section~\ref{s:pre}.

We prove the following theorem.

\begin{theorem*} 
Suppose $G$ is a Lie group associated with a graph $\cG$ and equipped with a left-invariant metric $g$. The following three statements are equivalent.
\begin{enumerate} [label=\emph{(\roman*)},ref=\roman*]
  \item \label{it:gom}
   The metric Lie group $(G, g)$ is a geodesic orbit manifold.

  \item \label{it:nrm}
  The metric Lie group $(G, g)$ is a naturally reductive manifold.

  \item \label{it:Ksemi}
  The graph $\cG$ is the disjoint union of complete graphs \emph{(}and so $G$ is the direct product of two-step free groups and an abelian group\emph{)} and $g$ is a semi-standard metric.
\end{enumerate}
\end{theorem*}

Note that we consider isolated vertices as complete graphs on a single vertex. We also note that if the graph $\cG$ is connected and $n=|\cE| \ne 4$, then the metric $g$ in \eqref{it:Ksemi} of the Theorem is in fact proportional to a standard metric. Indeed, in that case we have $\cJ=\so(\ag)$ which is simple for $n > 4$ and $n=3$ and is one-dimensional for $n=2$.

\section{Preliminaries and example}
\label{s:pre}


Let $(G,g)$ be a simply connected two-step nilmanifold equipped with a left-invariant metric $g$. Let $\n=\Lie(G)$ and let $\ipr$ be the corresponding inner product on $\n$. As in Section~\ref{s:intro}, denote $\z = [\n, \n], \; \ag = \z^\perp$, introduce the operators $J_Z, \; Z \in \z$, by \eqref{eq:defJZ} and denote $\cJ \subset \so(\ag)$ their span. For $D \in \End(\n)$ define $D_\ag \in \End(\ag)$ by $D_\ag X = \pi_\ag DX$, where $X \in \ag$ and $\pi_\ag$ is the orthogonal projection to $\ag$.

The following description of the algebra $\Der(\n)$ of derivations of $\n$ is well known (note that any derivation maps $\z$ to $\z$ and that any endomorphism $D$ mapping $\n$ to $\z$ and $\z$ to zero is a derivation).

\begin{lemma} \label{l:der}
Let $D \in \End(\n)$. Then
\begin{enumerate} [label=\emph{(\alph*)},ref=\alph*]
  \item \label{it:dergen}
  $D \in \Der(\n)$ if and only if for all $Z \in \z$ we have $J_ZD_\ag+D_\ag^*J_Z=J_{D^*Z}$, where the asterisk denotes the metric conjugation.

  \item \label{it:derskew} 
  $D \in \Der(\n) \cap \so(\n)$ if and only if both $\ag$ and $\z$ are $D$-invariant and for all $Z \in \z$ we have $J_{DZ}=[D_\ag,J_Z]$.
\end{enumerate}
\end{lemma}
Note that every derivation $D$ of $\n$ which keeps $\ag$ and $\z$ invariant is uniquely determined by $D_\ag$: the restriction of $D$ to $\z$ can be found from Lemma~\ref{l:der}\eqref{it:dergen} as $J$ is injective, and so to find all such derivations it suffices to find all $T \in \End(\ag)$ such that $J_ZT + T^*J_Z \in \cJ$.

Geodesic orbit and naturally reductive nilmanifolds are characterised as follows (recall that any naturally reductive homogeneous Riemannian manifold is geodesic orbit).

\newpage

\begin{proposition} \label{prop:gordon}
Let $(G, g)$ be a nilmanifold with a left-invariant metric.
\begin{enumerate} [label=\emph{(\alph*)},ref=\alph*]
  \item \label{it:twostep} If $(G, g)$ is geodesic orbit manifold, then $G$ is two-step \cite[Theorem~2.2]{G}.

  \item \label{it:gogordon} $(G, g)$ is a geodesic orbit manifold if and only if, for every $X \in \ag$ and $Z \in \z$ there exists $D \in \Der(\n) \cap \so(\n)$ such that $DZ = 0$ and $DX = J_Z X$ \cite[Theorem~2.10]{G}.

  \item \label{it:nrgordon} $(G, g)$ is a naturally reductive manifold if and only if $\cJ$ is a subalgebra of $\so(\ag)$ and $J^{-1} \circ \ad_{J_Z} \circ J \in \so(\z)$, for every $Z \in \z$ \cite[Theorem~2.5]{G}.
\end{enumerate}

\end{proposition}
Note that by Lemma~\ref{l:der}\eqref{it:derskew} and since $J$ is injective, the condition $DZ = 0$ from \eqref{it:gogordon} is equivalent to the fact that $[D_\ag, J_Z]=0$. For further results on the metric Lie algebra of a geodesic orbit nilmanifold we refer the reader to \cite[Section~2]{G} and \cite[Proposition~3.2]{GN}.

\medskip

Let $\n$ be a two-step nilpotent algebra associated with a graph $\cG$. Choose a standard basis $\{e_1, \dots, e_n, z_1, \dots, z_m\}$ for $\n$ and equip $\n$ with the standard inner product corresponding to that basis. Let $M_{ij} \in \End(\ag)$ be defined by $M_{ij}e_i=e_j$ and $M_{ij}e_k =0$ for $k \ne i$. We have $\z=\Span(z_1, \dots, z_m), \; \ag=\Span(e_1, \dots, e_n)$ and $J_{z_\a} = M_{ij} - M_{ji}$, where $E_\a=V_iV_j$ and $i<j$.

Following \cite{DM}, for $i=1, \dots, n$ we denote $\Omega'_i = \{j \, | \, V_iV_j \in \cE\}$ and $\Omega_i = \Omega'_i \cup \{i\}$. The relation $\preceq$ on the set $\{1,2, \dots, n\}$ defined by $i \preceq j \; \Leftrightarrow \; \Omega'_i \subset \Omega_j$ is transitive: if $i \preceq j$ and $j \preceq k$, then $i \preceq k$. Then the relation $i \sim j \; \Leftrightarrow \; (i \preceq j \; \& \; j \preceq i)$ is an equivalence relation; moreover, the relation $\preceq$ descends to the set of equivalence classes where it becomes a partial order.

The algebra $\Der(\n)$ is characterised by Lemma~\ref{l:der}\eqref{it:dergen} and the following Proposition.
\begin{proposition}[{\cite[Theorem~4.2]{DM}}] \label{prop:dm}
$\{D_\ag \, | \, D \in \Der(\n)\} = \Span(M_{ij} \, | \, i \preceq j)$.
\end{proposition}

We finish this section with an example of a semi-standard but not standard inner product on a nilpotent algebra associated with a graph.

\begin{example*} Consider the graph $\cG$ with $12$ vertices and $11$ edges shown in Figure~\ref{fig:graph}.

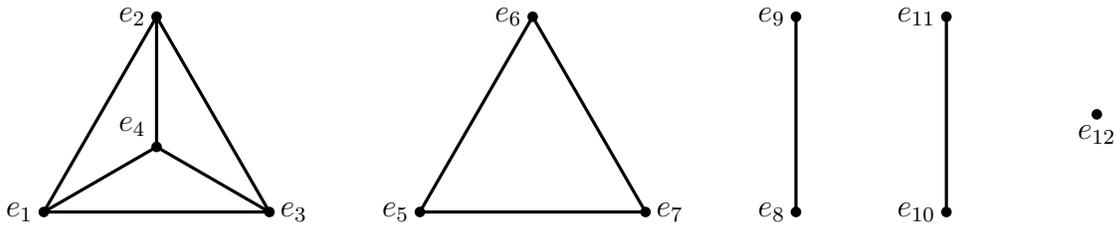
\begin{figure}[h]
\begin{tikzpicture}
\node[coordinate] (E1) at (0,0) [label=180:$e_1$] {};
\node[coordinate] (E3) at (3,0) [label=0:$e_3$] {};
\node[coordinate] (E2) at (1.5, {1.5*sqrt(3)}) [label=180:$e_2$] {};
\node[coordinate] (E4) at (1.5,{0.5*sqrt(3)}) [label=160:$e_4$] {};
\node[coordinate] (E5) at (5,0) [label=180:$e_5$] {};
\node[coordinate] (E7) at (8,0) [label=0:$e_7$] {};
\node[coordinate] (E6) at (6.5, {1.5*sqrt(3)}) [label=180:$e_6$] {};
\node[coordinate] (E8) at (10,0) [label=180:$e_8$] {};
\node[coordinate] (E9) at (10,{1.5*sqrt(3)}) [label=180:$e_9$] {};
\node[coordinate] (E10) at (12,0) [label=180:$e_{10}$] {};
\node[coordinate] (E11) at (12,{1.5*sqrt(3)}) [label=180:$e_{11}$] {};
\node[coordinate] (E12) at (14,{0.75*sqrt(3)}) [label=-90:$e_{12}$] {};
\foreach \x in {E1,E2,E3,E4,E5,E6,E7,E8,E9,E10,E11,E12} {\fill (\x) circle (2.0pt);}
\draw [very thick] (E1) -- (E2) -- (E3) -- (E4) --(E1); \draw [very thick] (E4) -- (E2); \draw [very thick] (E1) -- (E3);
\draw [very thick] (E5) -- (E6) -- (E7) -- (E5);
\draw [very thick] (E8) -- (E9); \draw [very thick] (E10) -- (E11);
\end{tikzpicture}
\caption{The graph $\cG = K_4 \sqcup K_3 \sqcup K_2 \sqcup K_2 \sqcup K_1$.}
\label{fig:graph}
\end{figure}
With the labelling of the vertices as in Figure~\ref{fig:graph} we get a standard basis $\{e_i, z_\a\}$ for the Lie algebra $\n$ associated with $\cG$ and the following bracket relations:
\begin{align*}
  [e_1, e_2]=z_1, & & [e_1, e_4]=z_4, & & [e_5, e_6]=z_7, & & [e_8, e_9]=z_{10}, \\
  [e_2, e_3]=z_2, & & [e_2, e_4]=z_5, & & [e_6, e_7]=z_8, & & [e_{10}, e_{11}]=z_{11},\\
  [e_1, e_3]=z_3, & & [e_3, e_4]=z_6, & & [e_5, e_7]=z_9.
\end{align*}
Let $\ip$ be the standard inner product on $\n$ which corresponds to that standard basis. Then $\cJ = \so(4) \oplus \so(3) \oplus \so(2) \oplus \so(2)$ and we have a $B$-orthogonal decomposition $\cJ = \ig_1 \oplus \ig_2 \oplus \ig_3 \oplus \ig_4 \oplus \ig_5$ into the sum of irreducible ideals given by
\begin{gather*}
\ig_1 \! =\Span(J_1+J_6, J_2+J_4, J_5-J_3) \simeq \so(3), \ig_2 \! =\Span(J_1-J_6, J_2-J_4, J_5+J_3) \simeq \so(3), \\
\ig_3=\Span(J_7, J_8, J_9) \simeq \so(3), \; \ig_4= \br (\cos \theta J_{10} + \sin \theta J_{11}), \; \ig_5= \br (-\sin \theta J_{10} + \cos \theta J_{11}),
\end{gather*}
for an arbitrary $\theta \in \br$ (where we abbreviated $J_{z_i}$ to $J_i$). The subspaces $\z_k = J^{-1} \ig_k$ are given by
\begin{gather*}
\z_1=\Span(z_1+z_6, z_2+z_4, z_5-z_3), \; \z_2=\Span(z_1-z_6, z_2-z_4, z_5+z_3), \\
\z_3=\Span(z_7, z_8, z_9), \; \z_4= \br (\cos \theta z_{10} + \sin \theta z_{11}), \; \z_5= \br (-\sin \theta z_{10} + \cos \theta z_{11}).
\end{gather*}
We can define a family of semi-standard inner products $\ipr$ on $\n$ such that the decomposition $\n=\ag \oplus \z_1 \oplus \z_2 \oplus \z_3 \oplus \z_4 \oplus \z_5$ is orthogonal, $\ipr_{|\ag} = \ip_{|\ag}$, and $\ipr_{|\z_k} = \lambda_k \ip_{|\z_k}$, where $\lambda_k > 0$ for $k=1, \dots, 5$.
\end{example*}

\section{Proof of the Theorem}
\label{s:proof}

We start with proving the implication \eqref{it:gom} $\Rightarrow$ \eqref{it:Ksemi}. 

Suppose $\n$ is a Lie algebra associated with a graph $\cG$ and equipped with an inner product $\ipr$ such that $(\n,\ipr)$ generates a geodesic orbit metric $g$ on $G$. Define $\z$ and $\ag$ as in Section~\ref{s:intro}. For any choice of a standard basis $\{e_1, \dots, e_n, z_1, \dots, z_m\}$, we have $\z=\Span(z_1, \dots, z_m)$, and then adding any linear combinations of the $z_\a$'s to the $e_i$'s we get another standard basis, so we can choose a standard basis such that $\ag=\Span(e_1, \dots, e_n)$ (later in the proof, we will specify the standard basis further). Let $\ip$ be the standard inner product corresponding to the chosen standard basis and let $A \in \End(\ag)$ and $C \in \End(\z)$ be defined by $(AX,Y) = \<AX,Y\>$ and $(CZ,W) = \<CZ,W\>$, for all $X, Y \in \ag, \; Z, W \in \z$. The endomorphisms $A$ and $C$ are symmetric relative to the restrictions of $\ip$ to $\ag$ and to $\z$ respectively and are positive definite.

For $Z \in \z$ we have
\begin{equation}\label{eq:JJ0}
J_Z = A^{-1} J^0_{CZ},
\end{equation}
where $J$ and $J^0$ are computed relative to the inner products $\ipr$ and $\ip$ respectively.

Now suppose the graph $\cG$ contains an edge $E_\a=V_i V_j$ such that $i \not\sim j$. We have $J^0_{z_\a}=\pm(M_{ij}-M_{ji})$. Without loss of generality we can assume that $i \not\preceq j$. Then by Proposition~\ref{prop:dm}, for any derivation $D$ of $\n$ we have $\<De_i, e_j\>=0$. Now let $Z= C^{-1}z_\a$, and $X=e_i$. We have $J_ZX = A^{-1} J^0_{CZ}X = A^{-1} J^0_{CC^{-1}z_\a}e_i = \pm A^{-1} e_j$ and so by Proposition~\ref{prop:gordon}\eqref{it:gogordon}, there exists $D \in \Der(\n)$ such that $De_i = \pm A^{-1} e_j$. But then $0=\<De_i, e_j\> = \pm \<A^{-1} e_j, e_j\>$, in contradiction with the fact that $A$ is positive definite.

It follows that $V_iV_j$ is an edge of $\cG$ only when $i \sim j$. But by \cite[Remark~4.6]{DM}, in every $\sim$-equivalence class, either all the vertices are pairwise adjacent, or no two vertices are adjacent, so that for every equivalence class, the induced subgraph is either complete or empty (edgeless). Therefore $\cG$ is the disjoint union of complete graphs (some of which can be isolated vertices). All the isolated vertices form a single equivalence class $\cC_0$. Denote $\cC_\mu, \; \mu = 1, \dots, p$, the other equivalence classes; all of them have at least two elements, and $\cG$ is the disjoint union of the complete graphs on the sets $\cC_\mu, \; \mu = 1, \dots, p$, and the isolated vertices from $\cC_0$.

We can now specify the standard basis further. The subspace $\ag_0=\Span(e_i \, | \, i \in \cC_0) \subset \ag$ is central in $\n$. It follows that adding arbitrary vectors from $\ag_0$ to the elements $e_i$ of a standard basis produces another standard basis. We can therefore assume that our standard basis is chosen in such a way that the restrictions of the inner products $\ipr$ and $\ip$ to $\ag_0$ coincide and that the orthogonal complements to $\ag_0$ in $\ag$ relative to both $\ipr$ and $\ip$ are the same.

For $\mu =1, \dots, p$ define the subspaces $\ag_\mu=\Span(e_i \, | \, i \in \cC_\mu) \subset \ag$. Denote $\cJ^0:=\Span(J^0_Z \, | \, Z \in \z)$. We summarise the above in the following lemma.

{
\begin{lemma} \label{l:orth}
The direct decomposition $\ag = \ag_0 \oplus \ag_1 \oplus \dots \oplus \ag_p$ is orthogonal relative to $\ip$. We have $\cJ^0=\oplus_{\mu=1}^p \so(\ag_\mu, \ip)$ and $\{D_\ag \, | \, D \in \Der(\n)\} = \oplus_{\mu=0}^p \End(\ag_\mu)$.
\end{lemma}
\begin{proof}
The first two assertions follow from the construction of the standard basis. The third one, from Proposition~\ref{prop:dm}.
\end{proof}
}

We will need the following linear algebraic fact. Let $r=\max (\rk J_Z \, | \, Z \in \z)$. Then there is an open, dense subset $\mU \in \z$ such that $\rk J_Z = r$ for all $Z \in \mU$ (as the condition $\rk J_Z < r$ is equivalent to a system of polynomial equations for $Z$). We say that $Z \in \z$ is \emph{generic} if $\rk J_Z = r$ and all the nonzero (complex) eigenvalues of $J_Z$ are pairwise distinct.

{
\begin{lemma}\label{l:generic}
There exist an open, dense subset $\mU' \subset \mU \subset \z$ consisting of generic vectors.
\end{lemma}
\begin{proof}
First note that it is sufficient to find at least one generic $Z \in \z$. Indeed, the characteristic polynomial $\chi_Z(\lambda)$ of $J_Z$ has the form $\chi_Z(\lambda)=f_Z(\lambda) \lambda^{n-r}$, where $f_Z$ is a polynomial of degree $r$ whose coefficients depend polynomially on $Z$ and whose constant term is nonzero for $Z \in \mU$. The fact that $J_Z$ has a nonzero multiple eigenvalue is equivalent to the fact that the discriminant of $f_Z$ is zero which is equivalent to vanishing of a certain polynomial in $Z$. If the complement to the zero set of that polynomial is a non-empty subset of $\z$, then it is open and dense.

To construct a generic $Z$ we note that by \eqref{eq:JJ0} the operator $J_Z$ is conjugate to the operator $SJ^0_{CZ}S$, where $S$ is a positive definite, $\ip$-symmetric square root of $A^{-1}$. Note that $SJ^0_{CZ}S$ is skew-symmetric relative to $\ip$. Furthermore, the maximal rank $r$ of $J_Z, \; Z \in \z$, equals the maximal rank of $J^0_Z, \; Z \in \z$, which by Lemma~\ref{l:orth} equals $\sum_{\mu=1}^p r_\mu$, where $r_\mu=2\lfloor\frac12\dim \ag_\mu \rfloor$. Now for every $\mu=1, \dots, p$, choose operators $K_{\mu,j} \in \so(\ag_\mu, \ip), \; j=1, \dots, \frac12 r_\mu$ as follows. Take an $\ip$-orthonormal basis $b_1, \dots, b_{\dim \ag_\mu}$ for $\ag_\mu$ and define $K_{\mu,j}$ so that the $(2j-1,2j)$-th entry of its matrix relative to that basis is $1$, the $(2j,2j-1)$-th entry is $-1$ and all the other entries are zeros. Then any linear combination of the operators $K_{\mu,j}$ which has $0 \le s \le \frac12 r_\mu$ nonzero coefficients has rank $2s$. Now extend every operator $K_{\mu,j}$ to an operator in $\so(\ag, \ip)$ by defining it to be zero on the orthogonal complement to $\ag_\mu$ relative to $\ip$, and then label the resulting operators consecutively, $K_1, \dots, K_r$. Note that $K_i \in \cJ^0$ and that the rank of any linear combination of the operators $K_i$ which has $0 \le s \le \frac12 r$ nonzero coefficients is $2s$.

Now the operator $S(x_1K_1)S, \; x_1 \ne 0$, has rank two, with two nonzero eigenvalues $\pm \lambda_1 \mathrm{i}$. The operator $S(x_1K_1+x_2K_2)S, \; x_1, x_2 \ne 0$, has rank four. As the eigenvalues depend continuously on the coefficients, for small enough $x_2$ it has four distinct nonzero eigenvalues $\pm \lambda'_1 \mathrm{i}, \pm \lambda_2 \mathrm{i}$, where $\lambda'_1$ is close to $\lambda_1$ and $\lambda_2$ is close to zero. Repeating the argument we get a sequence of nonzero numbers $x_1, \dots, x_r \in \br$ such that the operator $S(x_1K_1+\dots+x_rK_r)S$ has rank $r$, with $r$ pairwise distinct eigenvalues. But $x_1K_1+\dots+x_rK_r \in \cJ^0$, so that $x_1K_1+\dots+x_rK_r=J^0_W$ for some $W \in \z$. By \eqref{eq:JJ0} we have $J_{C^{-1}W} = A^{-1}J^0_W = S(SJ^0_WS)S^{-1}$, and so the vector $C^{-1}W$ is generic.
\end{proof}
}

By Proposition~\ref{prop:gordon}\eqref{it:gogordon}, for any $Z \in \z$ and any $X \in \ag$, there exists $D=D(Z,X) \in \Der(\n)$ which keeps $\z$ and $\ag$ invariant, which is skew-symmetric relative to $\ipr$ and such that $J_ZX=D(Z,X)X$ and $D(Z,X)Z=0$. The latter condition is equivalent to the fact that $[J_Z,D_\ag(Z,X)]=0$. Now assume that $Z \in \mU'$. As $Z$ is generic, the centraliser of $J_Z$ in $\so(\ag, \ipr)$ is $\Span(J_Z, J_Z^3, \dots, J_Z^{r-1}) \oplus \kg$, where $\kg=\{Q \in \so(\ag, \ipr) \, | \, J_Z Q = 0\}$. It follows that $D_\ag(Z,X) = \a_1 (Z,X) J_Z + \a_3 (Z,X) J_Z^3 \dots + \a_{r-1} (Z,X) J_Z^{r-1} + Q(Z,X)$ for some $\a_1 (Z,X), \a_3 (Z,X), \dots, \a_{r-1} (Z,X) \in \br$ and some $Q(Z,X) \in \kg$. Now let $\ag^0 = \Ker J_z, \; \ag^1 = \Im J_Z$ and let $X = X^0 + X^1$, where $X^0 \in \ag^0$ and $X^1 \in \ag^1$. From $J_ZX=D_\ag(Z,X)X$ we obtain $J_Z X^1 = \a_1 (Z,X) J_Z X^1 + \a_3 (Z,X) J^3_Z X^1 + \dots + \a_{r-1} (Z,X) J_Z^{r-1}X^1 + Q(Z,X) X^0$. Projecting to $\ag^0$ we find that $Q(Z,X) X^0=0$. Let $\mU''(Z) \subset \ag$ be the set of those $X \in \ag$ for which the set of vectors $\{J_Z^i X \, | \, i=1, \dots, r-1\}$ $(=\{J_Z^i X^1 \, | \, i=1, \dots, r-1\})$ is linear independent. As $Z$ is generic, $\mU''(Z)$ is an open and dense subset in $\ag$ and is the complement to the union of the zero sets of a finite number of polynomials of $X$ whose coefficients are polynomials in $Z$ (representing the fact that $J_ZX \wedge J_Z^3X \wedge \dots \wedge J_Z^{r-1}X = 0$). It follows that the subset $\mU^\star=\cup_{Z \in \mU'}\mU''(Z) \subset \z \times \ag$ is open and dense in $\z \times \ag$ and moreover, for all pairs $(Z, X) \in \mU^\star \subset \z \times \ag$ we have
\begin{equation}\label{eq:JD}
\begin{gathered}
  D_\ag(Z,X) = J_Z + Q(Z,X), \quad\text{where } \\ Q(Z,X) \in \so(\ag, \ipr), \quad J_Z Q(Z,X)=0, \quad Q(Z,X) X=0.
\end{gathered}
\end{equation}
Then for all such pairs, $J_Z D_\ag(Z,X)=J_Z^2$, and so by \eqref{eq:JJ0} we obtain $J^0_{CZ}D_\ag(Z,X) = J^0_{CZ}A^{-1}J^0_{CZ}$. Now suppose that $p > 1$. By Lemma~\ref{l:orth}, both $J^0_Z$ and $D_\ag$ keep the $\ip$-orthogonal decomposition $\ag=\oplus_{\mu=0}^p \ag_\mu$ invariant. Then for any $Y_1 \in \ag_\mu, \; Y_2 \in \ag_\nu$, $\mu \ne \nu, \; 1 \le \mu, \nu \le p$, we have $0 = \<J^0_{CZ}D(Z,X) Y_1, Y_2\> = \<J^0_{CZ}A^{-1}J^0_{CZ} Y_1, Y_2\> = \<A^{-1}J^0_{CZ} Y_1, J^0_{CZ} Y_2\>$. This is satisfied for an open, dense set of $Z \in \z$, hence for all $Z \in \z$. As $Y_1 \in \ag_\mu$ and $Y_2 \in \ag_\nu$ are arbitrary and we obtain $\<A^{-1} \ag_\mu, \ag_\nu\>=0$. Since we have already specified our standard basis in such a way that $A \ag_0 \subset \ag_0$, we obtain that all the subspaces $\ag_\mu, \; \mu=0,1, \dots, p$, are $A$-invariant.

We can now specify the standard basis $\{e_1, \dots, e_n, z_1, \dots, z_m\}$ further by noting that we can take for $\{e_1, \dots, e_n\}$ the union of \emph{any} bases for $\ag_\mu, \; \mu=0,1, \dots, p$, and then modify the vectors $z_1, \dots, z_m$ accordingly. It follows that we can choose a standard basis in such a way that $A=\id$, that is, such that the restrictions of $\ipr$ and of $\ip$ to $\ag$ coincide. Then from \eqref{eq:JD} and \eqref{eq:JJ0} we obtain that for all $(Z, X) \in \mU^\star$ we have $D_\ag(Z,X) = J^0_{CZ} + Q(Z,X)$, with $D_\ag(Z,X), J^0_{CZ}, Q(Z,X) \in \so(\ag, \ip)$ and  $J_{CZ} Q(Z,X)=0$. By Lemma~\ref{l:orth} and from the fact that $J_Z=J^0_{CZ}$ has the maximal rank (as $Z \in \mU'$) we obtain that $Q(Z,X) \in \so(\ag_0, \ip) \oplus 0_{|\oplus_{\mu=1}^p  \ag_\mu}$, so in particular, $J_W Q(Z,X) =0$ for all $W \in \z$.

We now consider $D_\z(Z,X)$, the restriction of $D(Z,X)$ to $\z$. It is skew-symmetric relative to the restriction of $\ipr$ to $\z$ and satisfies the equation in Lemma~\ref{l:der}\eqref{it:dergen}. As $A=\id$, we obtain from \eqref{eq:JJ0}:
\begin{gather}\label{eq:Dskewz}
  D_\z(Z,X)^t C \in \so(\z, \ip), \\
  J^0_{D_\z(Z,X)^tW} = [J^0_{W}, D_\ag(Z,X)], \quad \text{for all } W \in \z, \label{eq:Dskeder}
\end{gather}
and so for $(Z,X) \in \mU^\star$ equation \eqref{eq:Dskeder} gives $J^0_{D_\z(Z,X)^tW} = [J^0_{W}, J^0_{CZ}]$. From Lemma~\ref{l:orth}, $\cJ^0$ is a subalgebra; as the map $J^0: \z \to \cJ^0$ is injective, the above equation determines $D_\z(Z,X)$ uniquely. Moreover, by continuity the same equation holds for all $(Z, X) \in \z \times \ag$. Furthermore, $D_\z(Z,X)$ does not depend on $X$ (so that $D_\z(Z,X)=D_\z(Z)$) and the map $Z \mapsto D_\z(Z)$ is linear.

The bilinear form $B(K_1, K_2) = -\frac12\Tr (K_1K_2)$ introduced in Section~\ref{s:intro} is an invariant form on the algebra $\cJ^0$ and  $\<Z, W\>= B(J^0_Z, J^0_W)$, so \eqref{eq:Dskewz} gives $B(J^0_{D_\z(Z)^t C W}, J^0_W)=0$, for all $Z, W \in \z$, and so $0=B([J^0_{CW}, J^0_{CZ}], J^0_W) = B([J^0_W, J^0_{CW}], J^0_{CZ})$. As $C$ is nonsingular we get $[J^0_W, J^0_{CW}]=0$. Define the endomorphism $\Phi$ of the algebra $\cJ^0$ by $\Phi J^0_W = J^0_{CW}$. Note that $\Phi$ is symmetric relative to $B$, as $C$ is symmetric relative to $\ip$. Then we get $[K, \Phi K]=0$ for all $K \in \cJ^0$. It follows that $[K_1, \Phi K_2] = [\Phi K_1, K_2]$, and so $B([K_1, \Phi K_2], K_3) = B([\Phi K_1, K_2], K_3)$, for all $K_1, K_2, K_3 \in \cJ^0$. This gives $B([\Phi K_2, K_3], K_1) = B([K_2, K_3], \Phi K_1) = B(\Phi [K_2, K_3], K_1)$, and so $[\Phi, \ad_K]=0$ for all $K \in \cJ^0$. It follows that the eigenspaces of $\Phi$ are ideals of $\cJ^0$ which are orthogonal relative to $B$. Then $(Z, W) = \<CZ, W\> = B(J^0_{CZ}, J^0_W) = B(\Phi J^0_Z, J^0_W)$, so the inner product $\ipr$ is semi-standard, as required.

\medskip

As \eqref{it:nrm} $\Rightarrow$ \eqref{it:gom} is always true, to complete the proof we have to establish the implication \eqref{it:Ksemi} $\Rightarrow$ \eqref{it:nrm}. Let $\cG$ be the union of complete graphs and let $\ipr$ be a semi-standard inner product on the algebra $\n$ associated with $\cG$. We denote $\cC_\mu, \; \mu=0,1, \dots,p$, the equivalence classes of the vertices of $\cG$, where the class $\cC_0$ contains all the isolated vertices (and can be empty) and the induced subgraph on each of the classes $\cC_\mu, \; \mu=1, \dots, p$, is a complete graph on at least two vertices. We take the standard basis $\{e_1, \dots, e_n, z_1, \dots, z_m\}$, the corresponding standard inner product $\ip$ and the operator $\Phi$ as in the definition of the semi-standard inner product in Section~\ref{s:intro}. Note that the restrictions of $\ip$ and $\ipr$ to $\ag=\Span(e_1, \dots, e_n)$ coincide and that the derived algebra $\z=\Span(z_1, \dots, z_m)$ is the orthogonal complement to $\ag$ relative to both $\ipr$ and $\ip$. We have an orthogonal decomposition $\ag=\oplus_{\mu=0}^p \ag_\mu$ relative to the both inner products, where as above, $\ag_\mu=\Span(e_i \, | \, i \in \cC_\mu)$. By \eqref{eq:defJZ} we have $\cJ=\oplus_{\mu=1}^p \so(\ag_\mu)$ (for both inner products), and so $\cJ \subset \so(\ag)$ is a subalgebra, hence satisfying the first condition of Proposition~\ref{prop:gordon}\eqref{it:nrgordon}. Furthermore, for any $Z, W \in \z$ we have $((J^{-1} \circ \ad_{J_Z} \circ J)W, W) = (J^{-1} [J_Z, J_W], W) = B(\Phi [J_Z, J_W], J_W)$ by definition of the semi-standard inner product. As $B$ is invariant and $\Phi$ is symmetric relative to $B$  we get $B(\Phi [J_Z, J_W], J_W)=B([J_W, \Phi J_W], J_Z)$. But $[J_W, \Phi J_W]=0$ as the eigenspaces of $\Phi$ are $B$-orthogonal ideals of $\cJ$. So the second condition of Proposition~\ref{prop:gordon}\eqref{it:nrgordon} is satisfied, and hence the nilmanifold $(G,g)$ with $\n = \Lie(G)$ and the left-invariant metric $g$ generated by $\ipr$ is naturally reductive. This completes the proof of the Theorem.



\begin{thebibliography}{99}


\bibitem{Arv}
A.~Arvanitoyeorgos,
\emph{Homogeneous manifolds whose geodesics are orbits. Recent results and some open problems}, Irish Math. Soc. Bulletin,
\textbf{79} (2017), 5--29.

\bibitem{DM}
S.~Dani, M.~Mainkar,
\emph{Anosov automorphisms on compact nilmanifolds associated with graphs}, Trans. Amer. Math. Soc. \textbf{357} (2005), 2235--2251.

\bibitem{G}
C.~Gordon,
\emph{Homogeneous manifolds whose geodesics are orbits}, in: Topics in Geometry in Memory of Joseph D'Atri, Birkh\"{a}user, Basel, 1996, 155--174.

\bibitem{GN}
C.~Gordon, Yu.~G.~Nikonorov, 
\emph{Geodesic orbit Riemannian structures on $\Rn$}, J. Geom. Phys. \textbf{134} (2018), 235--243.

\bibitem{KV}
O.~Kowalski, L.~Vanhecke, \emph{Riemannian manifolds with homogeneous geodesics}, Boll. Un. Math. Ital. B(7) \textbf{5} (1991), 189--246.

\bibitem{M2014}
M.~Mainkar,
\emph{Graphs and two-step Nilpotent Lie algebras}, Groups Geom. Dyn. \textbf{8} (2014), 1--11.

\bibitem{Nik2017}
Yu.~G.~Nikonorov, \emph{On the structure of geodesic orbit Riemannian spaces}, Ann. Glob. Anal. Geom. \textbf{52} (2017), 289--311.

\end{thebibliography}
\end{document}